\documentclass[journal]{IEEEtran}
\usepackage{amsmath,amssymb,amsthm,bm,booktabs,graphicx,mathtools}
\usepackage{cite}
\usepackage[hidelinks]{hyperref}
\newtheorem{theorem}{Theorem}

\newtheorem{remark}{Remark}
\newcommand{\R}{\mathbb{R}}
\newcommand{\dd}{\,\mathrm d}
\newcommand{\jump}[1]{\left[#1\right]}
\newcommand{\avg}[1]{\left\{#1\right\}}

\newcommand{\pd}[2]{\frac{\partial #1}{\partial #2}}
\title{A Cut Finite Element Method for Transient Thermal Simulation in Multi-material Electronic Packaging Structures}
\author{Hao Dong$^{1,*}$\\[4pt]
\small $^{1}$School of Mathematics and Statistics, Xidian University, Xi'an 710071, PR China\\
\small $^*$Corresponding author: donghao@mail.nwpu.edu.cn}
\begin{document}
\maketitle

\begin{abstract}
Advanced electronic packaging structures represent a key technological approach for extending Moore's Law. An electronic packaging structure consists of multiple materials with distinct properties, constituting a composite structure with complex spatial architecture. This paper develops an unfitted cut finite element method (CutFEM) for transient heat conduction simulation in multi-material electronic packages structures with complex material interfaces. In the proposed computational framework, the heterogeneous spatial geometric features of electronic packaging structures is embedded and characterized in a regular background mesh, thereby avoiding the generation of body-fitted meshes on complicated interfaces. Interface temperature and heat-flux transmission conditions are imposed by a coefficient-weighted symmetric Nitsche formulation. In addition, a ghost-penalty stabilization controls arbitrarily small intersections between the physical materials and background meshes. Furthermore, the semi-discrete and fully-discrete numerical schemes are proposed, and an explicit energy estimate is derived in detail. Finally, two-dimensional and three-dimensional electronic packaging structures containing substrate, molding compound, die, and solder balls are designed to validate its accuracy, efficiency, and scalability in simulating challenging transient thermal problems.
\end{abstract}

\section{Introduction}
The integrated circuit (IC) industry is a key component of the modern electronic information industry. Its level of technological development reflects a nation's core competitiveness and serves as a powerful driving force for national economic development and an essential safeguard for national security. At present, the development of the IC industry has entered the post-Moore's Law era. The IC industry has shifted from the traditional approach of merely reducing feature sizes to the use of advanced packaging technologies, especially three-dimensional (3D) packaging technology, to improve system performance. This is regarded as an important solution for extending Moore's Law \cite{R1}. Electronic packaging structures face extremely harsh dynamic thermal environments during service, with local heat flux densities reaching as high as 1 kW/cm$^2$ (kilowatt-level per square centimeter). Efficient and high-precision computational simulation of their transient heat conduction problems is one of the key fundamental scientific issues in large-scale integrated circuit design and is of great value for the optimization and design of advanced electronic packaging structures \cite{R2}.

Electronic packaging structures are typically composed of multiple materials such as the die, solder, substrate, and epoxy molding compound (EMC). Different constituent materials and components have different thermal conductivities and heat capacities, and relatively large temperature gradients usually exist near material interfaces \cite{R3,R4}. Traditional body-fitted finite elements require the mesh to conform to every material interface. When a packaging structure contains a large number of spherical solder joints, curved interfaces, or moving interfaces, mesh generation and remeshing increase the computational cost and lead to reduced accuracy and efficiency of numerical algorithms. This difficulty motivates unfitted discretizations, in which the physical geometry is represented independently of a fixed computational mesh. In 1971, Nitsche proposed a method for weakly imposing Dirichlet boundary conditions, which provided an important theoretical tool for subsequent unfitted discretizations \cite{R5}. In 2002, Hansbo and Hansbo applied Nitsche's method to elliptic interface problems, allowing material interfaces to pass through finite elements and obtaining optimal convergence results independent of the interface position \cite{R6}. To address the ill-conditioning caused by small cut cells, Burman proposed the ghost-penalty stabilization technique in 2010 \cite{R7}. Subsequently, Burman and Hansbo developed fictitious-domain finite element frameworks based on Lagrange multipliers and stabilized Nitsche methods \cite{R8}. In 2015, Burman, Claus, Hansbo, Larson, and Massing systematically formulated the CutFEM concept by combining unfitted background meshes, integration over cut cells, weak Nitsche constraints, level-set descriptions of geometry, and ghost-penalty stabilization \cite{R9}. Since then, CutFEM has since been extended to problems in fluid mechanics, solid mechanics, interface phenomena, fluid--structure interaction, two-phase flow, surface equations, moving domains, shape optimization, and multi-physics coupling problems \cite{R10,R11,R12,R13,R14,R15}. It has also developed toward higher-order discretization, adaptive computation, parameter-robust methods,  robust preconditioning, parallel strategies, and complex three-dimensional engineering applications \cite{R16,R17,R18,R19,R20,R21,R22}.

In this study, we establish a novel CutFEM framework for transient heat conduction problems of multi-material electronic packages. The innovations of this work are introduced as follows.
\begin{itemize}
\item[1)]
The transient thermal simulation of multi-material electronic packages is implemented based on CutFEM for the first time. The transient thermal simulation is successfully accomplished for complex multi-material electronic packaging structures. \item[2)]
The coefficient-weighted Nitsche technique and ghost-penalty stabilization are established for multi-material package structures, offering numerical stability and robustness.
\item[3)]
An explicit energy-norm error estimate is rigorously derived for CutFEM, ensuring simulation accuracy and simulation.
\end{itemize}

This study is organized as follows. Section II briefly introduces the time-dependent multi-material thermal model for electronic packaging structures. Section III elaborately presents the CutFEM computational framework for dynamic multi-material thermal model. Interface temperature and heat-flux transmission conditions are imposed by a coefficient-weighted symmetric Nitsche formulation, and a ghost-penalty stabilization controls arbitrarily small intersections between the physical materials and background meshes. Furthermore, corresponding numerical algorithm is presented in detail. Section IV derives the explicit energy estimate for the proposed CutFEM. Section V validates the accuracy and efficiency of CutFEM through several benchmark examples. Finally, Section IV summarizes the main contributions and conclusion of this work.
\section{MULTI-MATERIAL THERMAL MODEL}
Let the electronic packaging structure domain be $\Omega\subset\R^d$ with $d=2,3$. The domain $\Omega$ is partitioned into disjoint material subdomains $\Omega_i$ with $i=1,\ldots,N$, such that $\Omega=\bigcup_{i=1}^{N}\Omega_i$. In each subdomain, its temperature field $T_i(\bm{x},t)$ satisfies the following governing equation.
\begin{equation}
\begin{aligned}
&\rho_i(\bm{x})c_i(\bm{x})\frac{\partial T_i(\bm{x},t)}{\partial t}
-\nabla\cdot[k_i(\bm{x})\nabla T_i(\bm{x},t)]\\
&\;\;\quad\quad=Q_i(\bm{x},t),
\qquad (\boldsymbol{x},t)\in\Omega_i\times(0,\mathcal T),
\end{aligned}
\label{eq:model}
\end{equation}
where $\rho_i$, $c_i$, $k_i$ and $Q_i $ represent density, specific heat capacity, thermal conductivity, and internal  heat source, respectively.

Furthermore, the interface between material region $\Omega_i$ and material region $\Omega_j$ is defined as $\Gamma_{ij}=\partial\Omega_i\cap\partial\Omega_j$. It is assumed that on the interface $\Gamma_{ij}$, the ideal thermal contact condition is satisfied, that is,
\begin{equation}
\begin{aligned}
&[T]=T_i-T_j=0,\\
&[k\nabla T\cdot\boldsymbol{n}]=k_i\nabla T_i\cdot\boldsymbol{n}-k_j\nabla T_j\cdot\boldsymbol{n}=0.
\end{aligned}
\label{eq:interface}
\end{equation}
Moreover, the Dirichlet boundary condition is adopted on the external boundary $\Gamma_D$ of the electronic packaging structure, namely
\begin{equation}
T_i(\bm{x},t)=\widehat T(\bm{x}),\quad\text{on }\Gamma_D\times(0,\mathcal T).
\label{eq:boundary}
\end{equation}
And also, the initial temperature is set to
\begin{equation}
  T_i(\bm{x},0)=\widetilde T_{i}(\bm{x}),
  \qquad \boldsymbol{x}\in\Omega_i.
  \label{eq:initial}
\end{equation}

Furthermore, the material interfaces are described by the level set function. For example, the chip region can be expressed as
\begin{equation}
\Omega_{\rm chip}=\{\boldsymbol{x}:\phi_{\rm chip}(\boldsymbol{x})\!<\!0\},
\Gamma_{\rm chip}=\{\boldsymbol{x}:\phi_{\rm chip}(\boldsymbol{x})\!=0\!\}.
\end{equation}

\section{THEORETICAL DERIVATION AND NUMERICAL ALGORITHM}
\subsection{CutFEM space and variational formulas}
Let $\mathcal T_h$ be a regular triangular or tetrahedral background grid partitioned for the geometric region of the electronic packaging structure, and the physical region $\Omega$ is not required to coincide with the element boundaries $\mathcal T_h$. Moreover, the active grid, the set of cut elements, and the set of interior elements are defined respectively as
\begin{equation}
\mathcal{T}_h^\Omega
=\{K\in\mathcal{T}_h:K\cap\Omega\neq\varnothing\},
\end{equation}
\begin{equation}
\mathcal{T}_h^{\rm cut}=\{K\in\mathcal{T}_h:K\cap\Omega\neq\varnothing,\ K\cap\partial\Omega\neq\varnothing\},
\end{equation}
\begin{equation}
\mathcal{T}_h^{\rm int}=\{K\in\mathcal{T}_h:K\subset\Omega\},
\end{equation}
where $\mathcal{T}_h^\Omega=\mathcal{T}_h^{\rm int}\cup\mathcal{T}_h^{\rm cut}$.

Since the electronic packaging structure is a complex composite material structure, for the multi-material interface problem, the active grid is defined separately for each material region as follows.
\begin{equation}
\mathcal T_{h}^i=\{K\in\mathcal T_h:K\cap\Omega_i\neq\varnothing\}.
\end{equation}
For the two sides of the material interface, independent finite element space $V_{h}^i$ is adopted, such that
\begin{equation}
\begin{aligned}
&V_h=V_{h,1}\times V_{h,2}\times\cdots\times V_{h,N},\\
&V_h^i=\{v|_{\mathcal T_h^i}:v\text{ is piecewise polynomial}\}.
\end{aligned}
\end{equation}

Let $\{\varphi_a^i(\bm{x})\}_{a\in\mathcal{N}_h^i}$ be the finite element basis functions of the space $V_h^i$, where
$\mathcal{N}_h^i$ is the set of active degree-of-freedom indices in the material region $\Omega_i$. In CutFEM,
$\varphi_a^i(\bm{x})$ can be understood as the basis function obtained by restricting the background grid shape function to the active grid $\mathcal T_{h}^i$ or the material subdomain $\Omega_i$, and $U_a^i(t)$ is the corresponding temperature degree of freedom. Thus, the discrete temperature can be written as
\begin{equation}
  T_h(\bm{x},t)
  =
  \sum_{i=1}^{N}
  \chi_i(\bm{x})
  \sum_{a\in\mathcal{N}_h^i}
  U_a^i(t)\varphi_a^i(\bm{x}),
  \label{eq:temperature-approx}
\end{equation}
where $\chi_i(\bm{x})$ is the indicator function of the material region $\Omega_i$ as below.
\begin{equation}
  \chi_i(\bm{x})=
  \begin{cases}
  1, & \bm{x}\in\Omega_i,\\
  0, & \bm{x}\notin\Omega_i.
  \end{cases}
\end{equation}

Based on the above definitions of the grid and finite element spaces, multiplying Eq. (1) by the test function $v_i\in V_h^i$ and integrating over $\Omega_i$ yields
\begin{equation}
  \int_{\Omega_i}
  \rho_i c_i \pd{T_i}{t}v_i\dd x
  \!-\!
  \int_{\Omega_i}
  \nabla\cdot(k_i\nabla T_i)v_i\dd x
  \!=\!
  \int_{\Omega_i}Q_i v_i\dd x.
\end{equation}
Then, performing integration by parts, it follows that
\begin{equation}
\begin{aligned}
&\int_{\Omega_i}  \rho_i c_i \pd{T_i}{t}v_i\dd x
  +  \int_{\Omega_i}  k_i\nabla T_i\cdot\nabla v_i\dd x\\
&-\int_{\partial\Omega_i}
  k_i\nabla T_i\cdot n_i\,v_i\dd s
  =
  \int_{\Omega_i}Q_i v_i\dd x.
\end{aligned}
\end{equation}
Since the trial and test functions in CutFEM are allowed to be discontinuous at the interface, additional interface consistency, symmetry, and penalty terms need to be constructed.

\subsection{Interface Nitsche terms}
In order to impose the boundary condition on the material interface, the jump is defined on the interface $\Gamma_{ij}$ as
\begin{equation}
  \jump{w}=w_i-w_j,
\end{equation}
and then define the weighted average heat flux as below
\begin{equation}
  \avg{k\nabla w\cdot \boldsymbol{n}}_{ij}
  =
  \omega_i k_i\nabla w_i\cdot \boldsymbol{n}
  +
  \omega_j k_j\nabla w_j\cdot \boldsymbol{n},
  \omega_i+\omega_j=1.
\end{equation}
For multi-material problems with large differences in thermal conductivity, one may take
\begin{equation}
  \omega_i=\frac{k_j}{k_i+k_j},
  \qquad
  \omega_j=\frac{k_i}{k_i+k_j}.
  \label{eq:weights}
\end{equation}
For any $T_h,v_h\in V_h$, the symmetric Nitsche term on the interface $\Gamma_{ij}$ is defined as
\begin{equation}
\begin{aligned}
&a_{\Gamma_{ij}}(T_h,v_h)
  =-\int_{\Gamma_{ij}}
  \avg{k\nabla T_h\cdot \boldsymbol{n}}_{ij}\jump{v_h}\dd s\\
&-\int_{\Gamma_{ij}}
  \avg{k\nabla v_h\cdot \boldsymbol{n}}_{ij}\jump{T_h}\dd s+\int_{\Gamma_{ij}}
  \eta_{ij}\jump{T_h}\jump{v_h}\dd s,
  \label{eq:interface-nitsche}
\end{aligned}
\end{equation}
where the first term ensures consistency with the original problem, the second term ensures the symmetry of the bilinear form, and the third term is a penalty term used to weakly enforce the temperature continuity condition. In formula (18), the penalty parameter is usually taken as
\begin{equation}
  \eta_{ij}
  =
  \gamma_\Gamma\frac{k_{ij}}{h_{\Gamma}},
  \qquad
  k_{ij}
  =
  \frac{2k_i k_j}{k_i+k_j},
  \label{eq:penalty-interface}
\end{equation}
where $h_{\Gamma}$ is the local mesh size near the interface, and $\gamma_\Gamma>0$ is a sufficiently large dimensionless parameter. In engineering computations, it is commonly taken as $\gamma_\Gamma=10\sim100$, and can be adjusted according to the element order, the interface cut shape, and the jumps in material parameters.

\subsection{Weak enforcement of outer boundary conditions}
If the outer boundary $\Gamma_D$ conforms to the background mesh, the outer boundary conditions can be imposed strongly. If the outer boundary is also described by level-set cutting, the Nitsche method can be used to weakly enforce the constant-temperature condition. Denote
\begin{equation}
  \Gamma_D^i=\Gamma_D\cap\partial\Omega_i.
\end{equation}
Furthermore, define the outer boundary Nitsche bilinear form as below
\begin{equation}
\begin{aligned}
  a_D(T_h,v_h)
  =
  \sum_{i=1}^{N}
  \bigg[
  -\int_{\Gamma_D^i}
  k_i\nabla T_{h,i}\cdot \boldsymbol{n}\,v_i\dd s\\
  -\int_{\Gamma_D^i}
  k_i\nabla v_i\cdot \boldsymbol{n}\,T_{h,i}\dd s+\int_{\Gamma_D^i}
  \eta_D T_{h,i}v_i\dd s
  \bigg],
  \label{eq:dirichlet-bilinear}
\end{aligned}
\end{equation}
and the corresponding right-hand side term is
\begin{align}
  \ell_D(v_h)
 \!=\!
  \sum_{i=1}^{N}
  \bigg[
  \!-\!\int_{\Gamma_D^i}
  k_i\nabla v_i\cdot \boldsymbol{n}\,\widehat T\dd s\!+\!\int_{\Gamma_D^i}
  \eta_D \widehat T v_i\dd s
  \bigg],
  \label{eq:dirichlet-linear}
\end{align}
where the boundary stabilization parameter $\eta_D$ is denoted as below
\begin{equation}
  \eta_D=\gamma_D\frac{k_i}{h}.
\end{equation}

\subsection{Ghost-penalty stabilization term}
In CutFEM, the interface may cut out very small material subregions $K_i$, which can lead to ill-conditioned mass or stiffness matrices. To improve the condition number and enhance stability on small cut cells, a ghost-penalty stabilization term is usually added on the interior faces near the cut cells.

Let $\mathcal{F}_{h,\Gamma}^{i}$ denote the set of interior faces in the interface neighborhood of material $\Omega_i$, where $F$ is one such interior face and $n_F$ is its unit normal vector. One may take
\begin{equation}
  g_h(T_h,v_h)
  =
  \sum_{i=1}^{N}
  \sum_{F\in\mathcal{F}_{h,\Gamma}^{i}}
  \sum_{m=1}^{p}
  \gamma_g k_i h_F^{2m-1}
  \int_F
  \jump{\partial_{n_F}^{m}T_{h,i}}
  \jump{\partial_{n_F}^{m}v_i}
  \dd s,
  \label{eq:ghost}
\end{equation}
where $\gamma_g>0$ is the stabilization parameter. This term remains consistent for smooth exact solutions and can suppress the numerical instability caused by small cut cells. $h_F$ denotes the local mesh size near the interior face $F$, and is usually taken as the average or the maximum of the diameters of the background elements adjacent to both sides of $F$, or approximately taken as the global mesh size $h$ on quasi-uniform meshes. The index $m$ denotes the order of the normal derivative, and $\partial_{\bm n_F}^{m}$ denotes the $m$-th order normal derivative along the unit normal vector
$\bm n_F$ of the interior face. The factor $h_F^{2m-1}$ is a scaling factor used to balance the derivative jump integral defined on the $(d-1)$-dimensional interior face, so that its dimension matches that of the main stiffness term. In particular, for linear elements with $p=1$, only the $m=1$ term is present, and the scaling factor is then $h_F$. This term remains consistent for smooth exact solutions and can suppress the numerical instability caused by small cut cells.

\subsection{Semi-discrete weak form}
Define the mass bilinear form
\begin{equation}
  m_h(T_h,v_h)
  =
  \sum_{i=1}^{N}
  \int_{\Omega_i}
  \rho_i c_i T_{h,i}v_i\dd x,
  \label{eq:mass}
\end{equation}
Moreover, define volume integral stiffness term
\begin{equation}
  a_\Omega(T_h,v_h)
  =
  \sum_{i=1}^{N}
  \int_{\Omega_i}
  k_i\nabla T_{h,i}\cdot\nabla v_i\dd x,
  \label{eq:volume-stiffness}
\end{equation}
and the total bilinear form is defined as follows
\begin{equation}
\begin{aligned}
&a_h(T_h,v_h)
  =
  a_\Omega(T_h,v_h)
  +a_\Gamma(T_h,v_h)
  \\
&+
  a_D(T_h,v_h)
  +
  g_h(T_h,v_h),
  \label{eq:ah}
\end{aligned}
\end{equation}
where $a_\Gamma(T_h,v_h)=\sum_{i<j}a_{\Gamma_{ij}}(T_h,v_h)$. In addition, define the right-hand side functional as
\begin{equation}
  \ell_h(v_h)
  =
  \sum_{i=1}^{N}
  \int_{\Omega_i}Q_i v_i\dd x
  +
  \ell_D(v_h).
  \label{eq:lh}
\end{equation}

Thus, the semi-discrete CutFEM scheme is: find $T_h(\bm{x},t)$ such that for any $v_h\in V_h$,
\begin{equation}
  m_h\left(\pd{T_h}{t},v_h\right)
  +
  a_h(T_h,v_h)
  =
  \ell_h(v_h).
  \label{eq:semi-discrete}
\end{equation}

\subsection{Fully discrete weak form}
Let $t_n=n\Delta t$, and let $T_h^n$ denote the numerical solution at time $t_n$. The following $\theta$-scheme is adopted:
\begin{equation}
  T_h^{n+\theta}
  =
  \theta T_h^{n+1}
  +
  (1-\theta)T_h^n,
  \qquad 0\leq\theta\leq1.
\end{equation}
Then the fully discrete scheme is: find $T_h^{n+1}\in V_h$ such that for any $v_h\in V_h$,
\begin{equation}
  m_h\left(
  \frac{T_h^{n+1}-T_h^n}{\Delta t},
  v_h
  \right)
  +
  a_h(T_h^{n+\theta},v_h)
  =
  \ell_h^{n+\theta}(v_h).
  \label{eq:theta}
\end{equation}
When $\theta=1$, it is the backward Euler scheme; when $\theta=0$, it is the forward Euler scheme; and when $\displaystyle\theta=\frac{1}{2}$, it is the Crank-Nicolson scheme.

\subsection{Numerical algorithm}
Based on the above theoretical analysis, the CutFEM numerical algorithm for the multi-material heat conduction problem of electronic packaging structures is established as follows:
\begin{enumerate}
\item[Step 1:] Generate a triangular or tetrahedral mesh in a regular background region containing the electronic packaging structure;
\item[Step 2:] Identify material regions, active elements, and cut elements according to the level-set function;
\item[Step 3:] Assemble the mass matrix, thermal conductivity matrix, Nitsche interface matrix, and ghost-penalty matrix;
\item[Step 4:] Advance time using the backward Euler method or other methods;
\item[Step 5:] Solve a sparse linear system at each time step;
\item[Step 6:] Output the temperature field of the electronic packaging structure at a fixed time.
\end{enumerate}

\section{ERROR ESTIMATION}
Based on the multi-material active mesh and finite element space defined in Section 3, this chapter presents an error analysis framework for the CutFEM solution of transient heat conduction problems in electronic packaging structures. Let \(\Omega=\bigcup_{i=1}^N\Omega_i\)£¬\(\Gamma_I=\bigcup_{i<j}\Gamma_{ij}\), and \(\mu_i(\bm{x})=\rho_i(\bm{x}) c_i(\bm{x})\), satisfying
\[
0<\mu_{\min}\leq\mu_i(\bm{x})\leq\mu_{\max},\\0<k_{\min}\leq k_i(\boldsymbol x)\leq k_{\max}<\infty.
\]

\subsection{Continuity, consistency, and energy norm}
For \(v=(v_1,\ldots,v_N)\in V_h\), define the energy norm as
\begin{equation}
\begin{aligned}
\|v\|_{E_h}^2&=\sum_{i=1}^N\|k_i^{1/2}\nabla v_i\|_{0,\Omega_i}^2
+\sum_{i<j}\frac{k_{ij}}{h_{ij}}\|[v]\|_{0,\Gamma_{ij}}^2\\
&+\sum_{i=1}^N\frac{k_i}{h_i}\|v_i\|_{0,\Gamma_D^i}^2
+g_h(v,v),
\label{eq:energy-norm-optimized}
\end{aligned}
\end{equation}
where \(h_{ij}\)ºÍ\(h_i\) are the local mesh sizes near the interface and the outer boundary, respectively. The mass norm is defined as
\begin{equation}
 \|v\|_{M_h}^2=\sum_{i=1}^N\mu_i\|v_i\|_{0,\Omega_i}^2.
 \label{eq:mass-norm-optimized}
\end{equation}

By the weighted trace inequality, for any \(\varepsilon>0\), we have
\begin{equation}
\begin{aligned}
&2\left|\int_{\Gamma_{ij}}\avg{k\nabla v\cdot\boldsymbol n}_{ij}[v]\,\dd s\right|\\
&\leq \varepsilon\sum_{r\in\{i,j\}}\|k_r^{1/2}\nabla v_r\|_{0,\omega_{ij}}^2+C_\varepsilon\frac{k_{ij}}{h_{ij}}\|[v]\|_{0,\Gamma_{ij}}^2,
\end{aligned}
\end{equation}
\begin{equation}
2\left|\!\int_{\Gamma_D^i}k_i\nabla v_i\cdot\boldsymbol n\,v_i\,\!\!\dd s\!\right|
\!\!\leq\!\! \varepsilon\|k_i^{1/2}\nabla v_i\|_{0,\omega_D^i}^2
\!+\!C_\varepsilon\frac{k_i}{h_i}\|v_i\|_{0,\Gamma_D^i}^2.
\end{equation}
By choosing \(\gamma_\Gamma\) and \(\gamma_D\) larger than the threshold determined by the trace constant, and using the ghost-penalty extension estimate, we obtain
\begin{equation}
 a_h(v_h,\!v_h)\!\!\geq\!\!c_0\|v_h\|_{E_h}^2,
 |a_h(w_h,\!v_h)|\!\!\leq\!\!c_1\|w_h\|_{E_h}\|v_h\|_{E_h},
 \label{eq:coercivity-continuity-optimized}
\end{equation}
where \(c_0,c_1>0\) are independent of the position of the interface relative to the background mesh; for a fixed material contrast, shape regularity, and range of penalty parameters, they are also independent of \(h\).

If \(T=(T_1,\ldots,T_N)\) is a piecewise smooth exact solution, then \([T]=0\) and \([k\nabla T\cdot\boldsymbol n]=0\) hold on \(\Gamma_{ij}\). Adding the interface integration-by-parts formulas from Section 3, the interface flux terms are exactly compensated by \(a_{\Gamma_{ij}}(T,v_h)\); similarly, the outer boundary terms are compensated by \(a_D(T,v_h)\) and \(\ell_D(v_h)\). Therefore, the following consistency relation is obtained:
\begin{equation}
 m_h(\frac{\partial T}{\partial t},v_h)+a_h(T,v_h)=\ell_h(v_h),
 \qquad \forall v_h\in V_h.
 \label{eq:consistency-optimized}
\end{equation}

\subsection{Elliptic projection and spatial approximation}
For each fixed time \(t\), define the elliptic projection \(R_hT(\bm{x},t)\in V_h\):
\begin{equation}
 a_h(R_hT(\bm{x},t)-T(\bm{x},t),v_h)=0,
 \qquad \forall v_h\in V_h.
 \label{eq:projection-optimized}
\end{equation}
Based on \eqref{eq:coercivity-continuity-optimized}, piecewise \(H^2\) regularity, and the extension operator on the active mesh, if \(T_i(\bm{x},t)\in H^2(\Omega_i)\), then
\begin{align}
 \|T-R_hT\|_{E_h}
 &\leq C h\left(\sum_{i=1}^N\|T_i\|_{H^2(\Omega_i)}^2\right)^{1/2},
 \\
 \|T-R_hT\|_{M_h}
 &\leq C h^2\left(\sum_{i=1}^N\|T_i\|_{H^2(\Omega_i)}^2\right)^{1/2}.
 \label{eq:projection-estimate-optimized}
\end{align}
Here, the second estimate requires the dual problem to have piecewise \(H^2\) regularity; if there are corners, cusps, or non-smooth intersection curves on the interface or outer boundary, only a correspondingly reduced order can be guaranteed.

\subsection{Backward Euler error equation}
Let \(t_n=n\Delta t\), and adopt the backward Euler scheme from Section 3. Define the total error
\begin{equation}
\begin{aligned}
&e^n=T(\bm{x},t_n)-T_h^n=\eta^n+\xi^n,\\
&\eta^n=T(\bm{x},t_n)-R_hT(\bm{x},t_n),\\
&\xi^n=R_hT(\bm{x},t_n)-T_h^n.
\end{aligned}
\end{equation}
By subtracting the discrete scheme and the projection definition from \eqref{eq:consistency-optimized}, it follows that for any \(v_h\in V_h\):
\begin{equation}
\begin{aligned}
&m_h\left(\frac{\xi^{n+1}-\xi^n}{\Delta t},v_h\right)
 +a_h(\xi^{n+1},v_h)\\
&=m_h(\tau^{n+1},v_h)
 -m_h\left(\frac{\eta^{n+1}-\eta^n}{\Delta t},v_h\right),
 \label{eq:error-equation-optimized}
\end{aligned}
\end{equation}
where $\tau^{n+1}$ is defined as below
\begin{equation}
 \tau^{n+1}=\frac{\partial T(\bm{x},t_{n+1})}{\partial t}-
 \frac{T(\bm{x},t_{n+1})-T(\bm{x},t_n)}{\Delta t}.
\end{equation}

Next, the Taylor formula gives
\begin{equation}
 \Delta t\sum_{n=0}^{N-1}\|\tau^{n+1}\|_{M_h}^2
 \leq C\Delta t^2\int_0^{t_N}
 \|\frac{\partial^2T(\bm{x},s)}{\partial s^2}\|_{M_h}^2\,\dd s.
 \label{eq:time-consistency-optimized}
\end{equation}
Furthermore, taking \(v_h=\xi^{n+1}\) in (42) and using the following inequality
\[
2m_h(a-b,a)=\|a\|_{M_h}^2-\|b\|_{M_h}^2+\|a-b\|_{M_h}^2,
\]
From \eqref{eq:coercivity-continuity-optimized} and Young's inequality, we obtain
\begin{equation}
\begin{aligned}
&\|\xi^{n+1}\|_{M_h}^2\!\!-\!\|\xi^n\|_{M_h}^2
\!\!+\!\|\xi^{n+1}\!-\!\xi^n\|_{M_h}^2
\!\!+\!c_0\Delta t\|\xi^{n+1}\|_{E_h}^2\\
&\leq C\Delta t\|\tau^{n+1}\|_{M_h}^2
 +C\Delta t\left|
 \frac{\eta^{n+1}-\eta^n}{\Delta t}\right|_{M_h}^2.
 \label{eq:energy-step-optimized}
\end{aligned}
\end{equation}
Summing over \(n=0,\ldots,N-1\) and using \eqref{eq:projection-estimate-optimized} for the estimate of \(\eta_t\), we obtain
\begin{equation}
\begin{aligned}
&\|\xi^N\|_{M_h}^2
 +\Delta t\sum_{n=0}^{N-1}\|\xi^{n+1}\|_{E_h}^2\\
&\leq C\left[
 h^4\int_0^{t_N}\sum_{i=1}^N\|\frac{\partial T_i(s)}{\partial s}\|_{H^2(\Omega_i)}^2\,\dd s\right.\\
&\left.
 +\Delta t^2\int_0^{t_N}\|\frac{\partial^2T(\bm{x},s)}{\partial s^2}\|_{M_h}^2\,\dd s+\|\xi^0\|_{M_h}^2\right].
 \label{eq:xi-estimate-optimized}
\end{aligned}
\end{equation}
\begin{theorem}[Error estimate of the backward Euler CutFEM scheme]
Assume that the background mesh is shape-regular, \(\Gamma_{ij}\) and \(\Gamma_D^i\) are piecewise \(C^2\) surfaces, the active mesh in Section 3 satisfies the adjacency chain condition with uniformly bounded length, the ghost penalty covers the stabilization neighborhood of each cut cell, and the material coefficients satisfy the above upper and lower bounds. If
\begin{equation*}
\begin{aligned}
&T_i\in L^\infty(0,t_N;H^2(\Omega_i)),\\
&\frac{\partial T_i}{\partial t}\in L^2(0,t_N;H^2(\Omega_i)),\\
&\frac{\partial^2 T}{\partial t^2}\in L^2(0,t_N;L^2(\Omega)),
\end{aligned}
\end{equation*}
and \(\|\xi^0\|_{M_h}\leq Ch^2\), then
\begin{equation}
\begin{aligned}
&\|T(\bm{x},t_N)\!-\!T_h^N\|_{M_h}^2\!+\!\Delta t\sum_{n=0}^{N-1}\|T(\bm{x},t_{n+1})\!-\!T_h^{n+1}\|_{E_h}^2\\
&\leq C\left(h^4+\Delta t^2\right).
\label{eq:final-error-optimized}
\end{aligned}
\end{equation}
where the constant \(C\) is independent of the interface cut position and \(h\), but may depend on the material contrast, the interface regularity, the penalty parameters, and the corresponding norms of the exact solution.
\end{theorem}
\begin{proof}
The error equation \eqref{eq:error-equation-optimized} is obtained from the consistency relation. Taking the test function
\(v_h=\xi^{n+1}\) and using the mass norm identity, discrete coercivity, and Young's inequality, we obtain \eqref{eq:energy-step-optimized}. Summing over the time levels and then using the time truncation estimate \eqref{eq:time-consistency-optimized} and the projection estimate \eqref{eq:projection-estimate-optimized}, we obtain \eqref{eq:xi-estimate-optimized}. Finally, from
\[
 T(\bm{x},t_n)-T_h^n=\eta^n+\xi^n
\]
and the triangle inequality, combining with \(\|\eta^n\|_{M_h}\leq Ch^2\) and \(\|\eta^n\|_{E_h}\leq Ch\), we obtain \eqref{eq:final-error-optimized}.
\end{proof}
\begin{remark}
From Theorem 1, the time-accumulated error in the mass norm is \(\|T-T_h\|_{M_h}=O(h^2+\Delta t)\), and the time-accumulated error in the energy norm is \(\|T-T_h\|_{E_h}=O(h+\Delta t)\). If inexact cut integration, a stabilization term that does not cover the complete ghost layer, a non-smooth interface, or only piecewise \(H^1\) regularity is used, then the corresponding approximation order needs to be reduced, and \eqref{eq:final-error-optimized} cannot be used directly.
\end{remark}
\begin{remark}
From Theorem 1, the error estimate of the Crank-Nicolson CutFEM scheme is $\|T(\bm{x},t_N)-T_h^N\|_{M_h}^2
+\Delta t\sum_{n=0}^{N-1}\|T(\bm{x},t_{n+1})-T_h^{n+1}\|_{E_h}^2
\leq C\left(h^4+\Delta t^4\right)$, the time-accumulated error in the mass norm is \(\|T-T_h\|_{M_h}=O(h^2+\Delta t^2)\), and the time-accumulated error in the energy norm is \(\|T-T_h\|_{E_h}=O(h+\Delta t^2)\).
\end{remark}
\section{NUMERICAL EXPERIMENTS AND DISCUSSION}
This chapter verifies the effectiveness of the developed CutFEM for simulating heat conduction in electronic packaging structures through numerical experiments. All numerical experiments in this paper are carried out on the same HP mobile workstation, whose CPU type is an Intel(R) Core(TM) Ultra 7 255H processor with a main frequency of 2.00 GHz and 32 GB of memory. Since analytical or exact solutions can hardly be obtained for heat conduction problems in electronic packaging structures, a fine finite element solution on a refined mesh is used as the reference solution. Let the fine finite element reference solution be denoted by \(T_{\mathrm{FEM}}\), and the CutFEM solution by \(T_{\mathrm{CutFEM}}\). Define the error function $e_h(\bm{x},t)=T_{\mathrm{CutFEM}}(\bm{x},t)-T_{\mathrm{FEM}}(\bm{x},t)$, and then define the following relative
$L^2$ norm and $L^{\infty}$  norm errors to evaluate the numerical accuracy of CutFEM:
\begin{align}
E_{L^2}(t)=\frac{\|e_h\|_2}{\|T_{\mathrm{FEM}}\|_2},\;\; E_{L^\infty}(t)=\frac{\|e_h\|_\infty}{\|T_{\mathrm{FEM}}\|_\infty}.
\end{align}
\subsection{Two-dimensional electronic packaging structure}
This example computes the transient heat conduction behavior of a two-dimensional electronic packaging structure. The two-dimensional electronic packaging structure consists of four materials: substrate, epoxy molding compound, chip, and solder balls. The substrate is located in the lower part of the region, the epoxy molding compound fills the remaining background region, the chip is located above the solder ball array, and the solder balls are arranged between the substrate and the chip, as shown in Figure~\ref{fig:2d-geometry}. The two-dimensional computational domain is $\Omega=(x_1,x_2)=[0,12]\times[0,6]$, the substrate region is $x_2\leq2$, the chip region is $|x_1-6|\leq2$ and $2.76\leq x_2\leq3.96$, the centers of the three solder balls are $(4.4,2.38)$, $(6.0,2.38)$ and $(7.6,2.38)$, respectively, with a radius of $0.42$, and the remaining part is the molding compound region. Therefore, the lower end of the solder balls $x_2=2.00$ is in contact with the top surface of the substrate, and the upper end $x_2=2.76$ is in contact with the bottom surface of the chip.
\begin{figure}[htbp]
\centering
\includegraphics[width=\columnwidth]{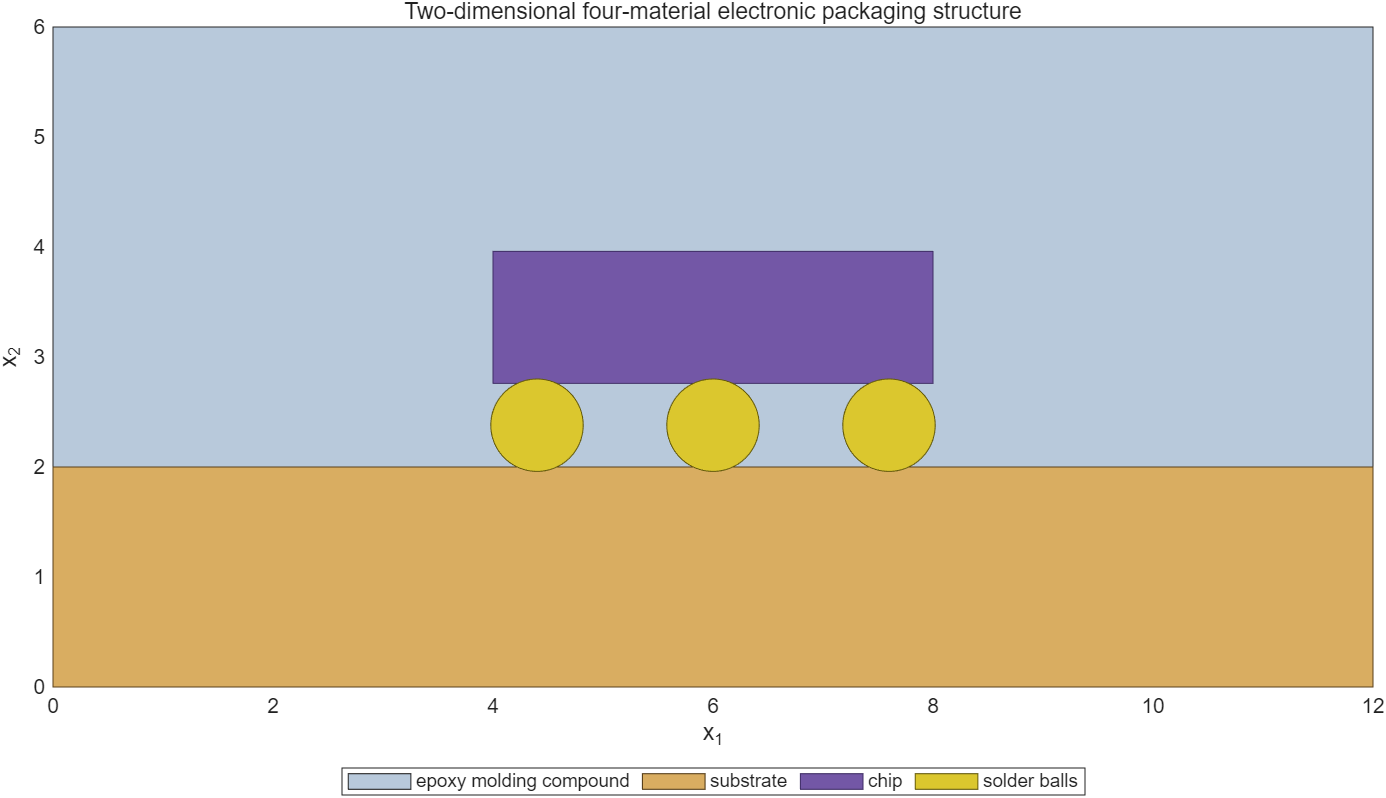}
\caption{Schematic diagram of the two-dimensional electronic packaging structure.}
\label{fig:2d-geometry}
\end{figure}

A internal heat source of 0.04 is applied to the chip region of the electronic packaging structure, and both the initial temperature and the outer boundary temperature are taken as 25. In addition, the parameters of the four materials¡ªsubstrate, epoxy molding compound, chip, and solder balls¡ªof the electronic packaging structure are shown in Table~\ref{tab:2d-material}.
\begin{table}[htbp]
\centering
\caption{Material parameters of each part of the electronic packaging structure.}
\label{tab:2d-material}
\begin{tabular}{cccc}
\toprule
Material & ID & $k_i$ & $\rho_i c_i$ \\
\midrule
Substrate & 1 & $0.003$ & $0.00160$ \\
Epoxy molding compound & 2 & $0.0008$ & $0.00130$ \\
Chip & 3 & $0.130$ & $0.00163$ \\
Solder balls & 4 & $0.050$ & $0.00150$ \\
\bottomrule
\end{tabular}
\end{table}

Next, based on the CutFEM method and algorithm established above, the heat conduction problem of the electronic packaging structure is solved. The temperature continuity condition and heat flux continuity condition on the material interfaces are weakly enforced using the multi-space Nitsche method, the ghost penalty near the interfaces is used to stabilize small cut cells, and the backward Euler scheme is used for time discretization. The computational time interval is uniformly taken as
$t\in[0,1]$, and the time step is set as $\Delta t=0.02$. For the two-dimensional computation, a regular Cartesian background mesh is used. The background rectangular mesh is first divided into $48\times24$ small rectangles along the $x_1$ and $x_2$ directions, respectively, and then each small rectangle is split into two triangles along the same diagonal. In addition, the mesh for refined FEM is obtained by dividing the domain into $192\times96$ small rectangles along the $x_1$ and $x_2$ directions, respectively, and then splitting each small rectangle into two triangles along the same diagonal. Figure~\ref{fig:mesh2d} shows the detailed CutFEM mesh and refined FEM mesh for two-dimensional electronic packaging structure.
\begin{figure}[!t]
\centering\includegraphics[width=\columnwidth]{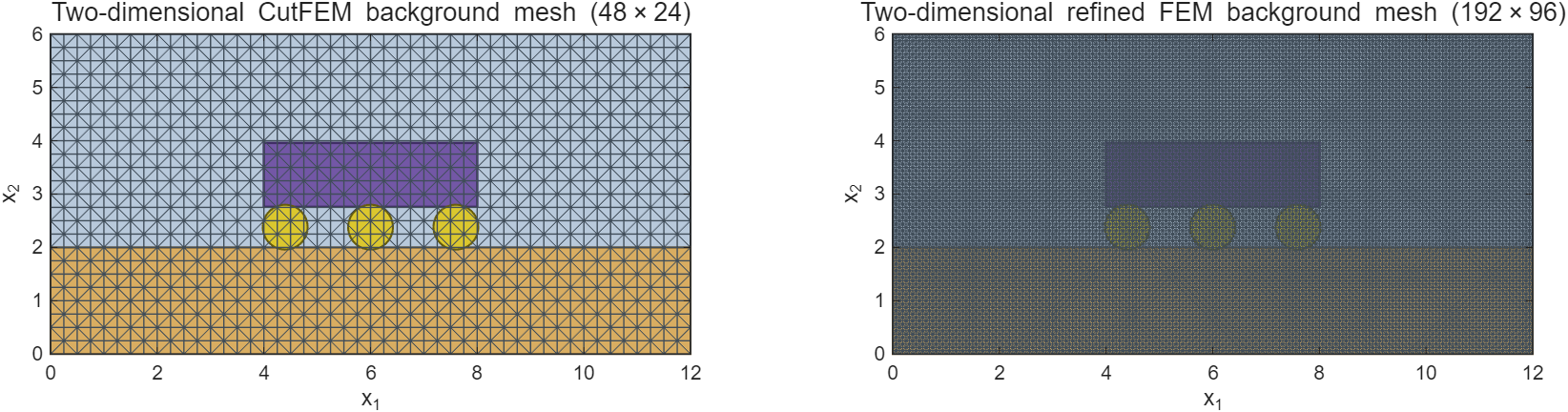}
\caption{CutFEM mesh and refined FEM mesh for the two-dimensional electronic packaging structure.}
\label{fig:mesh2d}
\end{figure}

After the numerical simulation is completed, Figure~\ref{fig:temperaturecompare} presents a comparison of the temperature fields at the final time $t=1.0$ computed by CutFEM and the fine FEM, Figure~\ref{fig:errorcompare} shows the absolute error distribution, Figure~\ref{fig:errorhistory} shows the evolutive relative errors of temperature field, and Figure~\ref{fig:parameter} presents an analysis of the influence of the Nitsche penalty parameter and the ghost-penalty parameter on the accuracy and efficiency of CutFEM. Table~\ref{tab:2dcompare} provides a comparison of the computational resource cost, numerical accuracy, and computational efficiency of the two methods.
\begin{figure}[!t]
\centering\includegraphics[width=\columnwidth]{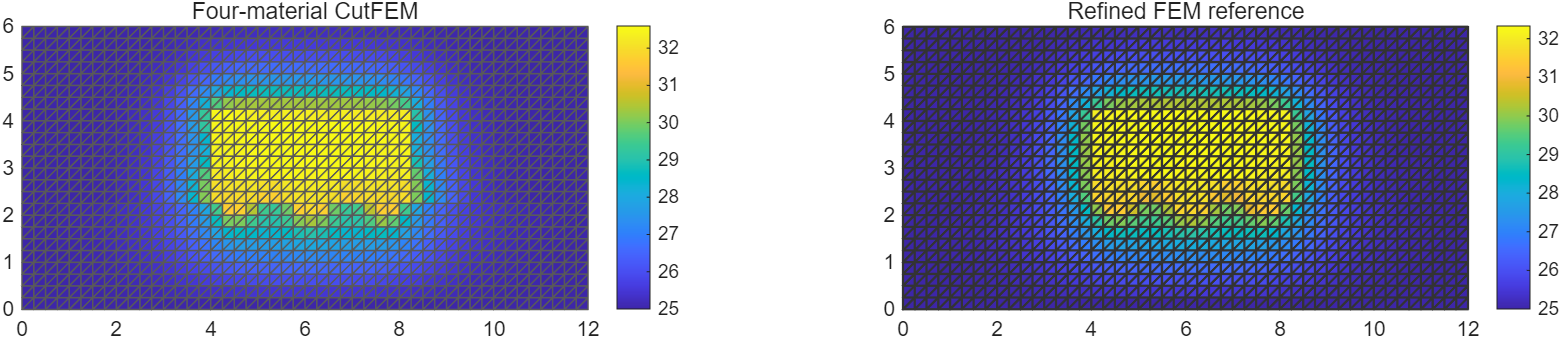}
\caption{Temperature fields at $t=1.0$ obtained by CutFEM and refined FEM.}
\label{fig:temperaturecompare}
\end{figure}

\begin{figure}[!t]
\centering\includegraphics[width=\columnwidth]{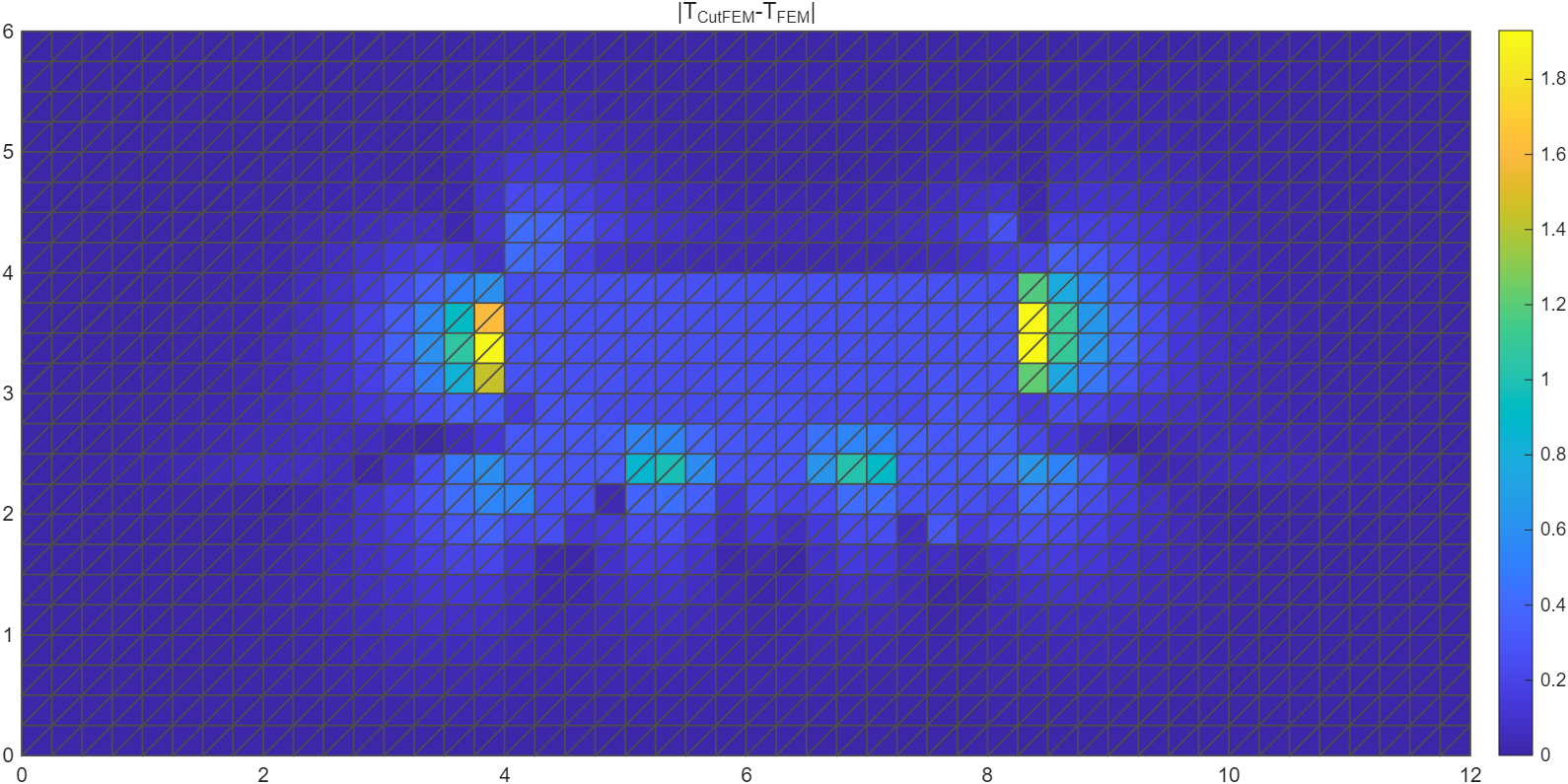}
\caption{Absolute value of the error between CutFEM and refined FEM.}
\label{fig:errorcompare}
\end{figure}

\begin{figure}[!t]
\centering\includegraphics[width=\columnwidth]{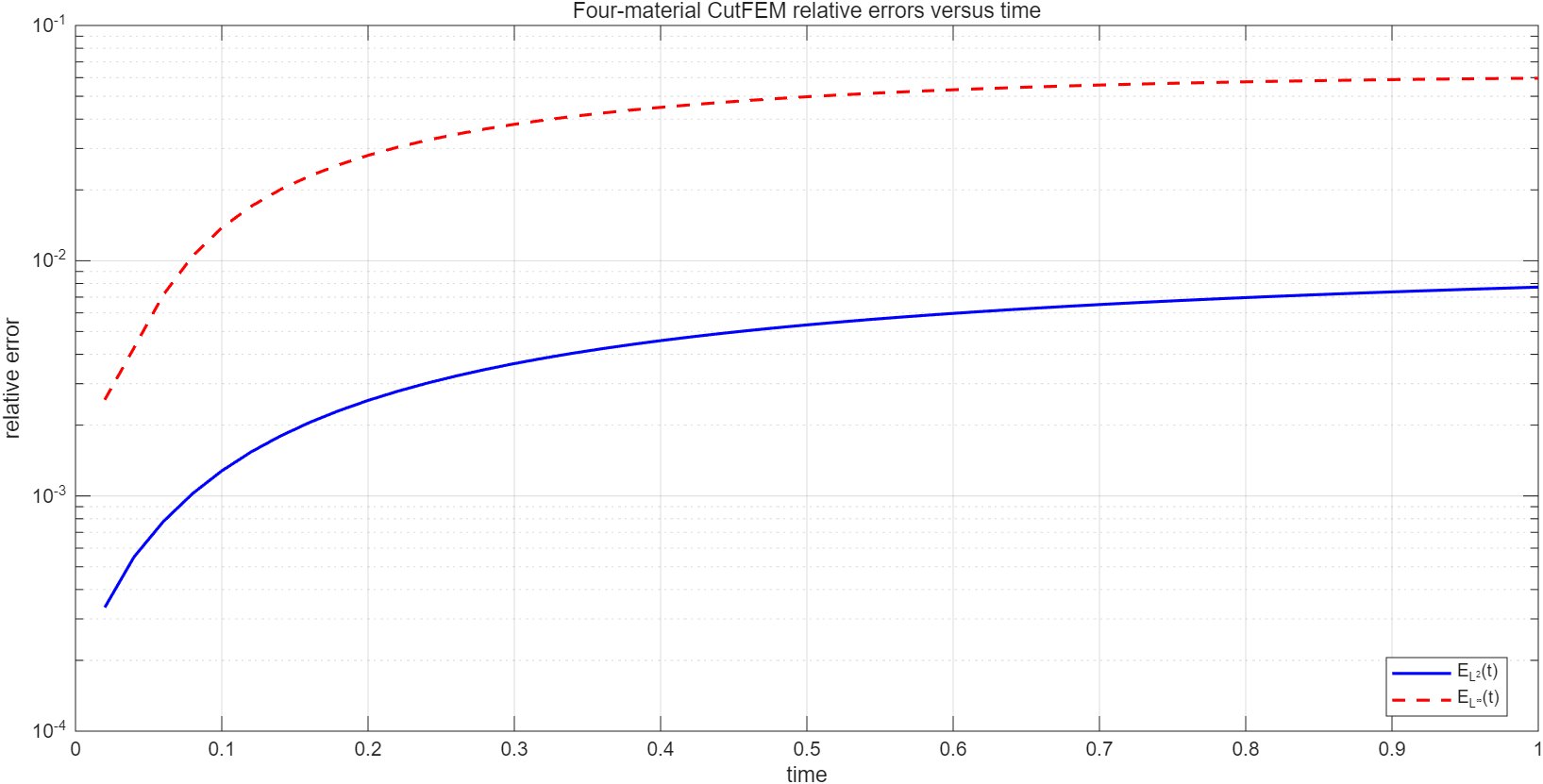}
\caption{The evolutive relative errors of temperature field.}
\label{fig:errorhistory}
\end{figure}

\begin{figure}[!htb]
\centering
\begin{minipage}[c]{0.48\textwidth}
  \centering
  \includegraphics[width=0.8\linewidth,totalheight=2.0in]{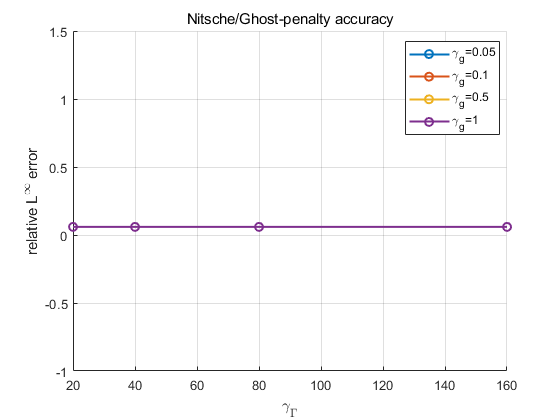}\\
  (a)
\end{minipage}
\begin{minipage}[c]{0.48\textwidth}
  \centering
  \includegraphics[width=0.8\linewidth,totalheight=2.0in]{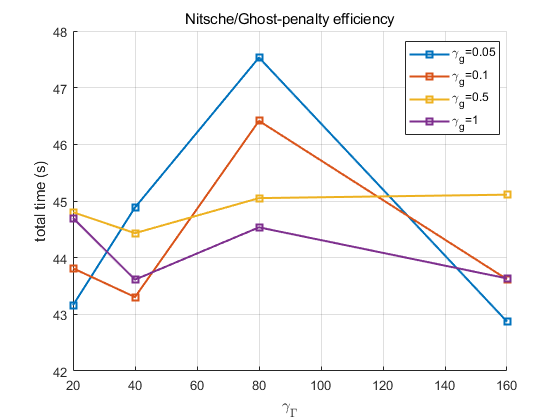}\\
  (b)
\end{minipage}
\caption{The effects of the Nitsche penalty function and the Ghost-penalty function on the accuracy and efficiency of CutFEM.}
\label{fig:parameter}
\end{figure}

\begin{table}[htbp]
\centering
\caption{Comparison between CutFEM and refined FEM.}
\label{tab:2dcompare}
\begin{tabular}{cccccc}
\toprule
Method & Elements & Dofs & $L^2$ error & $L^\infty$ error & Time \\
\midrule
CutFEM & $48\times24$ & 1401 & 0.7720\% & 5.9693\% & 43.92s \\
Refined FEM & $192\times96$ & 18721 & - & - & 73.08s \\
\bottomrule
\end{tabular}
\end{table}
From Figure~\ref{fig:temperaturecompare} and Figure~\ref{fig:errorcompare}, it can be seen that the numerical solutions obtained by the two computational methods both form a high-temperature region in the chip area and diffuse toward the substrate and molding compound, but there are still certain differences in the local high-temperature region and near the material interfaces. This is because CutFEM includes Nitsche terms and interior face stabilization terms, while the refined FEM method adopts strongly enforced outer boundary conditions and a continuous $P_1$ piecewise coefficient form. According to the results in Figure~\ref{fig:errorhistory}, it can be concluded that the proposed CutFEM is robust without blow-up
phenomenon for long-time thermal simulation. Figure~\ref{fig:parameter} illustrates that Nitsche and ghost-penalty parameters have limited impact on the accuracy and efficiency of CutFEM for 2D thermal simulation of multi-material package structure. In addition, from Table~\ref{tab:2dcompare}, it can be found that the mesh cost of CutFEM is far smaller than that of refined FEM, and the computational efficiency is also improved.

\subsection{Three-dimensional electronic packaging structure}
This example computes the transient heat conduction behavior of a three-dimensional electronic packaging structure. The three-dimensional electronic packaging structure consists of four materials: substrate, epoxy molding compound, chip, and solder balls. Its computational domain is $\Omega=(x_1,x_2,x_3)=[-7,7]\times[-5,5]\times[0,3.62]$. The top surface of the substrate is $x_3=0.70$, the solder balls have a radius of $0.34$ and their centers are located at $x_3=0.96$, so the lower ends of the solder balls $x_3=0.70$ are in contact with the substrate; the bottom surface of the chip is $x_3=1.22$, the upper ends of the solder balls $x_3=1.22$ are in contact with the chip, and the top surface of the chip is $x_3=1.87$. The molding compound fills the remaining region, as shown in Figure~\ref{fig:package3d}.
\begin{figure}[!t]
\centering\includegraphics[width=\columnwidth]{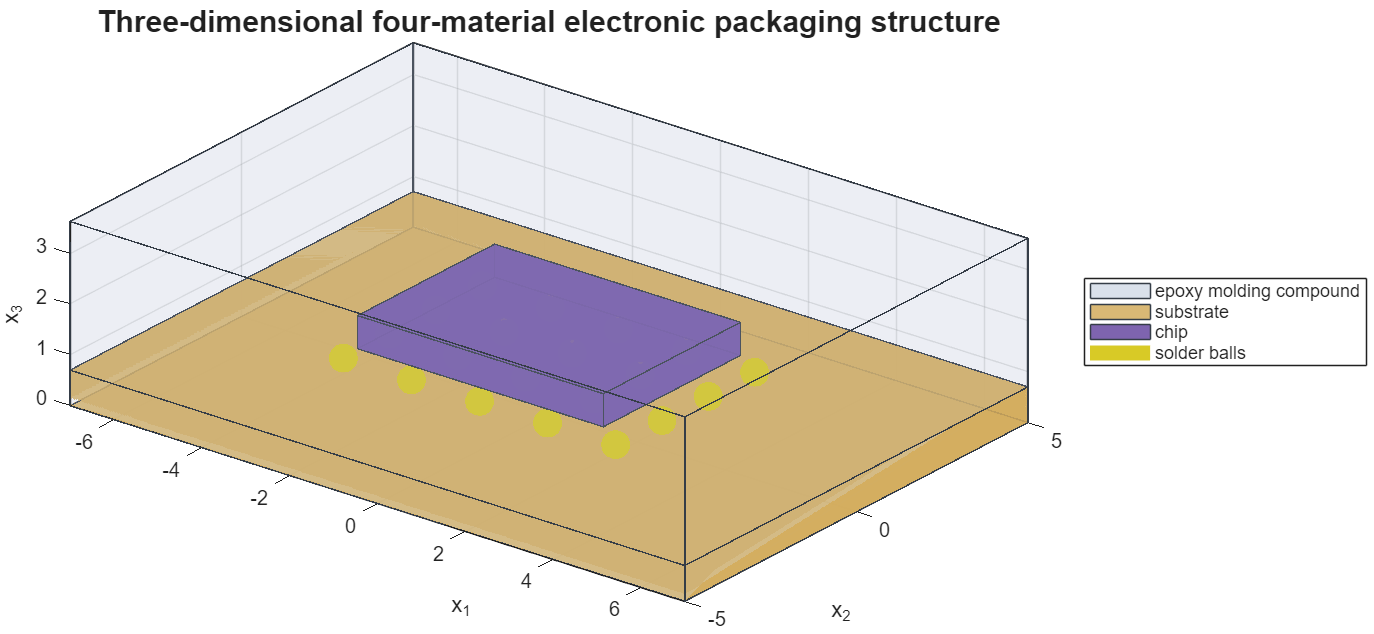}
\caption{Schematic diagram of the three-dimensional four-material electronic package.}
\label{fig:package3d}
\end{figure}

A internal heat source of $0.04$ is applied to the chip region of the three-dimensional electronic packaging structure, and both the initial temperature and the outer boundary temperature are taken as $25$. In addition, the four materials of the three-dimensional electronic packaging structure¡ªsubstrate, epoxy molding compound, chip, and solder balls¡ªuse the same material parameters as those in the two-dimensional example. Next, based on the CutFEM method and algorithm established above, the heat conduction problem of the electronic packaging structure is solved. The temperature continuity condition and heat flux continuity condition on the material interfaces are weakly enforced using the multi-space Nitsche method, the ghost penalty near the interfaces is used to stabilize small cut cells, and the backward Euler scheme is used for time discretization. The computational time interval is uniformly taken as $t\in[0,1]$, and the time step is set as
$\Delta t=0.02$. For the three-dimensional computation, a regular Cartesian background mesh is used. The background hexahedral mesh is first divided into 24, 18, and 12 small hexahedra along the $x_1$, $x_2$ and $x_3$ directions, respectively, and then each small hexahedron is split into six tetrahedra. In addition, the mesh for the fine finite element method is obtained by dividing the domain into 48, 36, and 24 small hexahedra along the $x_1$, $x_2$ and $x_3$ directions, respectively, and then splitting each small hexahedron into six tetrahedra. Figure~\ref{fig:mesh3d} shows the CutFEM mesh and refined FEM mesh.
\begin{figure}[!t]
\centering\includegraphics[width=\columnwidth]{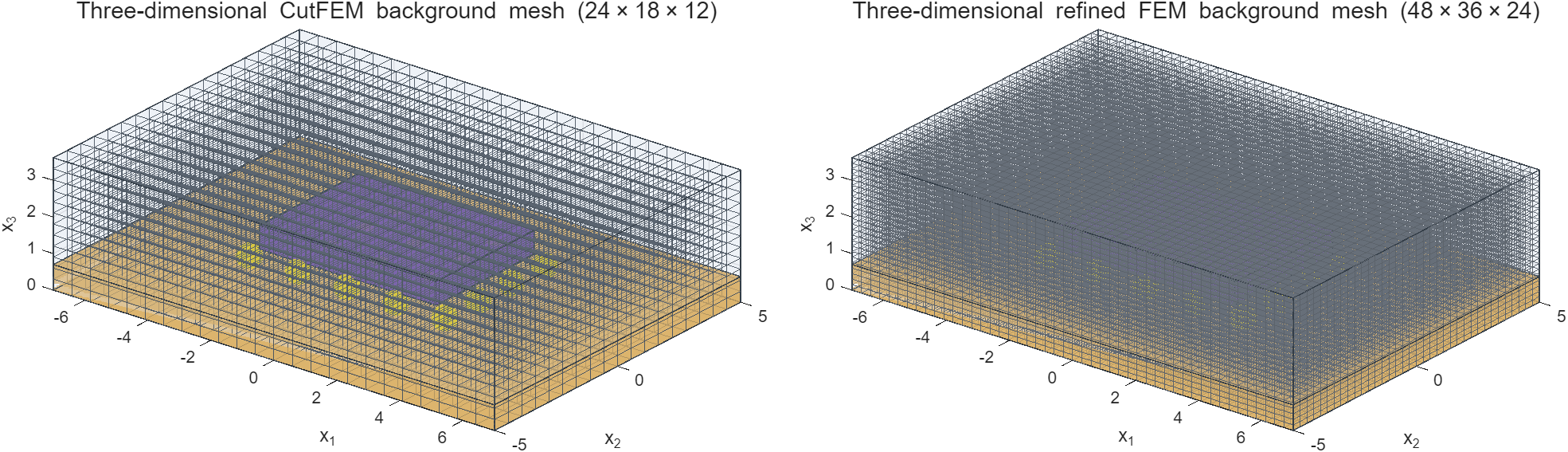}
\caption{Three-dimensional background tetrahedral mesh.}
\label{fig:mesh3d}
\end{figure}

After the numerical simulation is completed, Figure~\ref{fig:parametermesh3D} presents a comparison of the temperature fields computed by CutFEM and fine FEM at the final time $t=1.0$; Figure~\ref{fig:errorcompare3D} shows the absolute error distribution; Figure~\ref{fig:errorhistory3D} shows the evolutive relative errors of temperature field; Figure~\ref{fig:parameter3} presents an analysis of the effects of the Nitsche penalty function and the ghost-penalty function on the accuracy and efficiency of CutFEM. Table~\ref{tab:3dcompare} provides a comparison of the computational resource costs, numerical accuracy, and computational efficiency of the two methods.
\begin{figure}[!htb]
\centering
\begin{minipage}[c]{0.48\textwidth}
  \centering
  \includegraphics[width=0.98\linewidth,totalheight=2.4in]{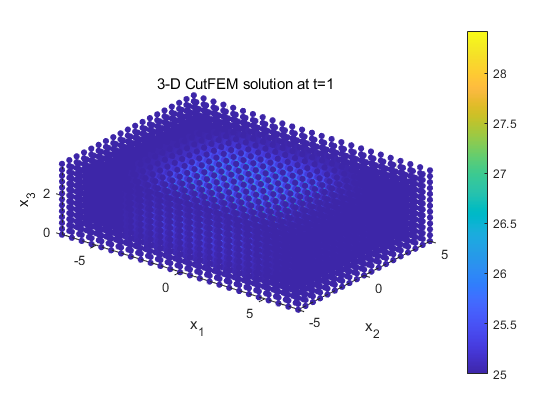}\\
  (a)
\end{minipage}
\begin{minipage}[c]{0.48\textwidth}
  \centering
  \includegraphics[width=0.98\linewidth,totalheight=2.4in]{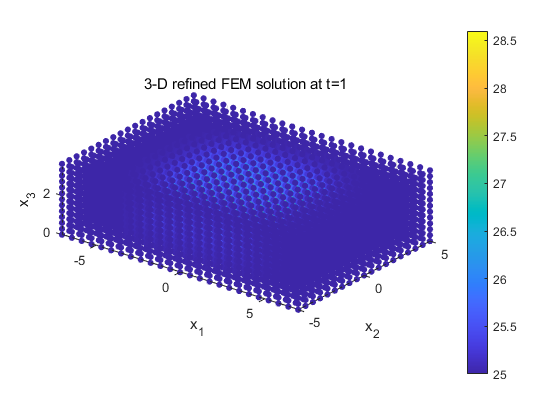}\\
  (b)
\end{minipage}
\begin{minipage}[c]{0.58\textwidth}
  \centering
  \includegraphics[width=0.98\linewidth,totalheight=1.0in]{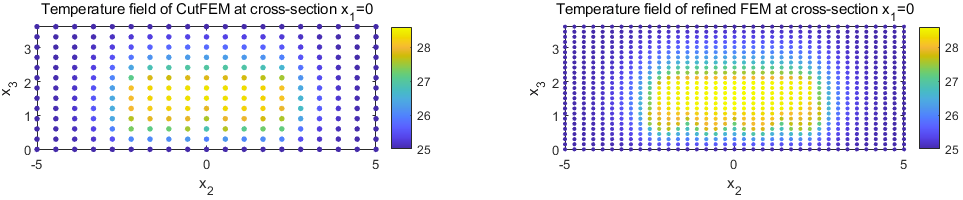}\\
  (c)
\end{minipage}
\begin{minipage}[c]{0.58\textwidth}
  \centering
  \includegraphics[width=0.98\linewidth,totalheight=1.0in]{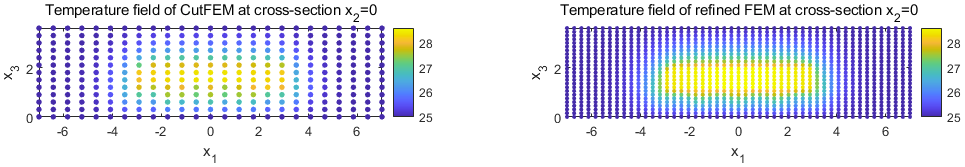}\\
  (d)
\end{minipage}
\caption{Temperature fields at $t=1.0$ obtained by CutFEM and refined FEM.}
\label{fig:parametermesh3D}
\end{figure}

\begin{figure}[htbp]
\centering
\includegraphics[width=\columnwidth]{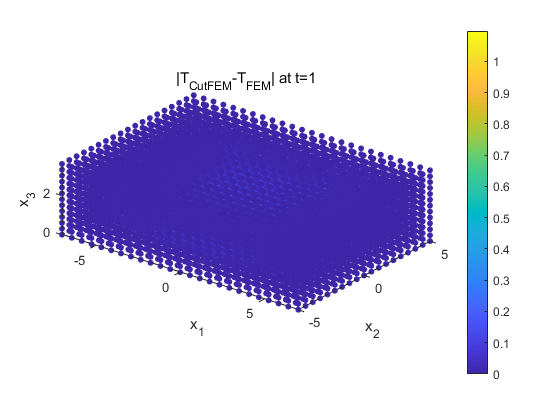}
\caption{Absolute value of the error between CutFEM and refined FEM.}
\label{fig:errorcompare3D}
\end{figure}

\begin{figure}[!t]
\centering\includegraphics[width=\columnwidth]{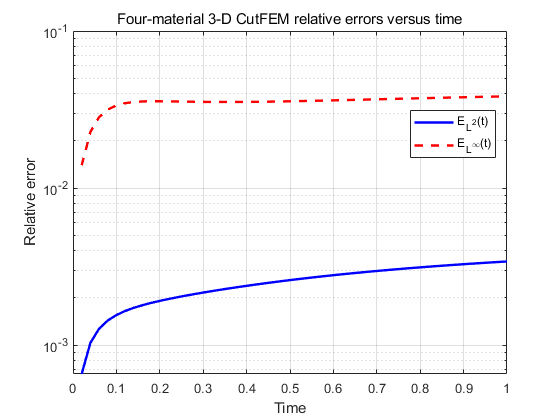}
\caption{The evolutive relative errors of temperature field.}
\label{fig:errorhistory3D}
\end{figure}

\begin{figure}[!htb]
\centering
\begin{minipage}[c]{0.48\textwidth}
  \centering
  \includegraphics[width=0.8\linewidth,totalheight=2.0in]{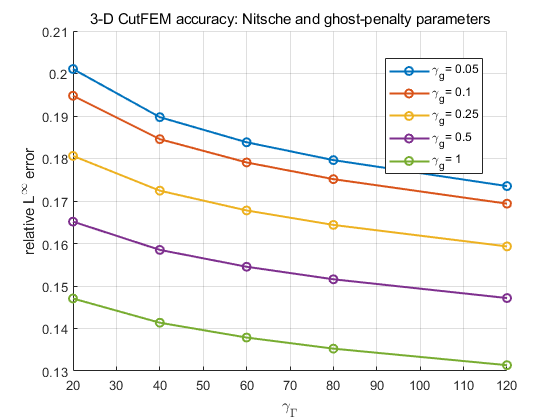}\\
  (a)
\end{minipage}
\begin{minipage}[c]{0.48\textwidth}
  \centering
  \includegraphics[width=0.8\linewidth,totalheight=2.0in]{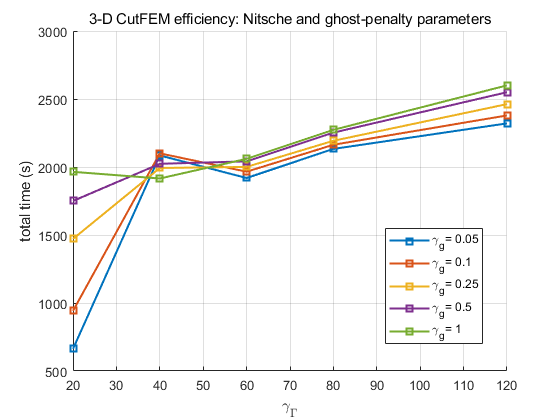}\\
  (b)
\end{minipage}
\caption{The effects of the Nitsche penalty function and the Ghost-penalty function on the accuracy and efficiency of CutFEM.}
\label{fig:parameter3}
\end{figure}

\begin{table}[htbp]
\centering
\caption{Comparison between CutFEM and refined FEM.}
\label{tab:3dcompare}
\begin{tabular}{cccccc}
\toprule
Method & Elements & Dofs & $L^2$ error & $L^\infty$ error & Time \\
\midrule
CutFEM & $24\times18\times12$ & 7878 & 0.3405\% & 3.8370\% & 286.83s \\
Refined FEM & $48\times36\times24$ & 45325 & - & - & 1653.84s \\
\bottomrule
\end{tabular}
\end{table}

It can be seen from Figure~\ref{fig:parametermesh3D} and Figure~\ref{fig:errorcompare3D} that the numerical solutions obtained by both computational methods form a high-temperature region in the chip area and diffuse toward the substrate and the molding compound. However, certain differences still exist in the local high-temperature region and near the material interfaces. This is because CutFEM includes Nitsche terms and interior edge stabilization terms, whereas the revised FEM adopts enforced outer boundary conditions and a continuous $P_1$ piecewise coefficient formulation. According to the results in Figure~\ref{fig:errorhistory3D}, it can be concluded that the proposed CutFEM is robust without blow-up
phenomenon for long-time thermal simulation. Figure~\ref{fig:parameter3} illustrates that Nitsche and ghost-penalty parameters have limited impact on the accuracy and efficiency of CutFEM for 3D thermal simulation of multi-material package structure. In addition, it can be found from Table~\ref{tab:3dcompare} that the mesh overhead of CutFEM is far smaller than that of the fine finite element method, and the computational efficiency is also significantly improved.

\section{CONCLUSIONS AND OUTLOOKS}
This paper establishes an effective CutFEM computational framework and its convergence theory for simulating complex heat conduction problems in electronic packaging structures. The main conclusions are as follows:
\begin{enumerate}
\item For multi-material electronic packaging structures with complex interfaces, CutFEM can represent the interfaces of chips, solder balls, substrates, and molding compounds on a regular background mesh, avoiding repeated generation of body-fitted meshes when the interfaces change;
\item The Nitsche method can effectively weakly enforce outer boundary and material interface conditions, and the ghost penalty can effectively improve stability when cut cells are too small;
\item Two-dimensional and three-dimensional numerical examples show that CutFEM can efficiently and accurately simulate complex heat conduction problems in electronic packaging structures.
\end{enumerate}

Based on the research on the CutFEM method for heat conduction problems in electronic packaging structures, future work will further study the CutFEM method and convergence theory for heat conduction problems with temperature-dependent thermal conductivity and interfacial thermal resistance, as well as for complete thermal-mechanical-electrical multi-field coupling problems. In addition, efficient preconditioning methods, model order reduction methods, and parallel algorithms suitable for complex electronic packaging structures will also be studied to further improve the numerical accuracy and computational efficiency of the CutFEM method.

\section*{Acknowledgments}
This work was supported by the National Natural Science Foundation of China (No.\hspace{1mm}12471387), Natural Science Basic Research Program of Shaanxi (No.\hspace{1mm}2026JC-YXQN-017), Xidian University Specially Funded Project for Interdisciplinary Exploration (No.\hspace{1mm}TZJH2024008), the Fundamental Research Funds for the Central Universities (No.\hspace{1mm}QTZX25082).

\end{document}